\documentclass[11pt]{article}
\usepackage{amsmath, amssymb, amsfonts, amsthm, latexsym}
\usepackage{mathrsfs}
\usepackage{graphicx}
\usepackage{authblk}
\usepackage[usenames]{color}
\newtheorem{theorem}{Theorem}[section]

\newtheorem{claim}[theorem]{Claim}

\newtheorem{proposition}[theorem]{Proposition}
\newtheorem{corollary}[theorem]{Corollary}
\newtheorem{remark}[theorem]{Remark}

\newtheorem*{notat*}{Notation}

\newcommand{\Bin}{\ensuremath{\textrm{Bin}}}

\def\eps{\varepsilon}

\newcommand\G{\mathcal{G}}

\title{Rainbow Hamilton cycles in randomly coloured\\ randomly perturbed dense graphs}

\author{Elad Aigner-Horev \thanks{Department of Computer Science, Ariel University, Ariel 40700, Israel. Email: {\tt horev@ariel.ac.il}.}
\quad Dan Hefetz \thanks{Department of Computer Science, Ariel University, Ariel 40700, Israel. Email: {\tt danhe@ariel.ac.il}. Research supported by ISF grant 822/18.}}

\begin{document}
\maketitle

\begin{abstract}
Given an $n$-vertex graph $G$ with minimum degree at least $d n$ for some fixed $d > 0$, the distribution $G \cup \mathbb{G}(n,p)$ over the supergraphs of $G$ is referred to as a (random) {\sl perturbation} of $G$. We consider the distribution of edge-coloured graphs arising from assigning each edge of the random perturbation $G \cup \mathbb{G}(n,p)$ a colour, chosen independently and uniformly at random from a set of colours of size $r := r(n)$. We prove that such edge-coloured graph distributions a.a.s. admit rainbow Hamilton cycles whenever the edge-density of the random perturbation satisfies $p := p(n) \geq C/n$, for some fixed $C > 0$, and $r = (1 + o(1))n$. The number of colours used is clearly asymptotically best possible. In particular, this improves upon a recent result of Anastos and Frieze (2019) in this regard. As an intermediate result, which may be of independent interest, we prove that randomly edge-coloured sparse pseudo-random graphs a.a.s. admit an almost spanning rainbow path. 
\end{abstract}

\section{Introduction} 

A classical result of Dirac~\cite{Dirac} asserts that every $n$-vertex graph $G$ (with at least three vertices) satisfying $\delta(G) \geq n/2$ is Hamiltonian. Moreover, Dirac's result is optimal as far as the constant $1/2$ appearing in the condition on $\delta(G)$ is concerned. 

Let $\G_{d,n}$ denote the set of $n$-vertex graphs with minimum degree at least $dn$ for some constant $d > 0$. As noted above, for every $d \in (0, 1/2)$, there are non-Hamiltonian graphs $G \in \G_{d,n}$. Nevertheless, Bohman, Frieze, and Martin~\cite{BFM03} discovered that once {\sl slightly} randomly perturbed (i.e. smoothed), the members of $\G_{d,n}$ almost surely give rise to Hamiltonian graphs. In particular, they proved that for every $d > 0$ there exists a $C := C(d)$ such that $G \cup \mathbb{G}(n,p)$ is a.a.s. Hamiltonian whenever $G \in \G_{d,n}$ and $p := p(n) \geq C/n$, undershooting the threshold for Hamiltonicity in $\mathbb{G}(n,p)$ by a logarithmic factor. Numerous results regarding spanning configurations in members of the graph distribution $\G_{d,n} \cup \mathbb{G}(n,p)$ (and its hypergraph analogue) have since appeared; see, e.g.,~\cite{BTW17, BHKM18, BFKM04, BHKMPP18, BMPP18, DRRS18, HMMMO, HZ18, KKS16, KKS17, MM18}.    

Define $\mathscr{G}(d,n,p,r)$ to be the distribution of edge-coloured supergraphs of members of $\G_{d,n}$ defined as follows. Fix $G \in \G_{d,n}$, perturb $G$ using $\mathbb{G}(n,p)$, and colour the edges of the perturbation $G \cup \mathbb{G}(n,p)$ by assigning each edge a colour chosen independently and uniformly at random from the set $[r]$.   

Anastos and Frieze~\cite{AF} proved that for every $d > 0$ there exists a $C := C(d)$ such that if $p := p(n) \geq C/n$ and $r > (120 - 20 \ln d) n$, then a.a.s. $\mathscr{G}(d,n,p,r)$ admits a {\sl rainbow}\footnote{A Hamilton cycle whose edges are coloured using $n$ distinct colours.} Hamilton cycle.
The corresponding problem for random graphs was extensively studied. In particular, improving earlier results, Frieze and Loh~\cite{FL} proved that randomly colouring the edges of $\mathbb{G}(n, (1 + o(1)) \ln n/n)$ with $(1 + o(1)) n$ colours a.a.s. yields a rainbow Hamilton cycle. Both the edge-density of the random graph and the number of colours asserted by their result are clearly asymptotically best possible; nevertheless, both were refined. Ferber and Krivelevich~\cite{FK} improved the result of Frieze and Loh by replacing $p = (1 + o(1)) \ln n/n$ with the optimal $p = (\ln n + \ln \ln n + \omega(1))/n$. Bal and Frieze~\cite{BF16} proved that if $p = \omega(\ln n/n)$, then precisely $n$ colours suffice. Ferber~\cite{F15} improved the latter result to $p = \Omega(\ln n/n)$.    

Our main result asserts that in the perturbed setting with random perturbations of edge-density $C/n$, a set of colours of size $(1 + o(1)) n$ suffices in order to yield a rainbow Hamilton cycle asymptotically almost surely. This improves upon the aforementioned result of Anastos and Frieze~\cite{AF} in terms of the number of colours used, which is clearly asymptotically best possible.  

\begin{theorem} \label{th::rainbowHAM}
For every $d, \alpha > 0$ there exists a constant $C := C(d, \alpha)$ such that whenever $p := p(n) \geq C/n$ and $r = (1 + \alpha) n$, the distribution of edge-coloured graphs $\mathscr{G}(d,n,p,r)$ a.a.s. admits a rainbow Hamilton cycle.  
\end{theorem}

%\noindent
%{\bf Remarks.}
%\begin{enumerate}
%	\item Our proof of Theorem~\ref{th::rainbowHAM} supports taking $d := d(n) = 
%	o(1)$, at the price of allowing $C := C(n)$ to grow to infinity. This bridges the gap, so to speak, between rainbow Hamiltonicity of dense graphs and the %aforementioned results of Frieze and Loh~\cite{FL} and of Ferber and Krivelevich~\cite{FK} regarding rainbow Hamilton cycles in random graphs. 
\begin{remark}	
It would be interesting to know whether Theorem~\ref{th::rainbowHAM} can be extended to handle $(1 + \alpha) n$ colours, where $\alpha := \alpha(n) = o(1)$ (in a meaningful manner) while $C$ remains fixed. Our current proof does not allow this. In particular, it would be interesting to know the smallest value of $C = C(n)$ for which the result holds with $\alpha = 0$. The aforementioned result of Ferber~\cite{F15} implies that $C = \Omega(\ln n)$ suffices.   
\end{remark}
%\end{enumerate}

\section{Almost spanning rainbow paths in sparse pseudorandom graphs}

\subsection{Rainbow DFS}

In this section, we put forth an adaptation of the well-known DFS algorithm, to which we refer as {\em rainbow DFS} (RDFS, hereafter), fit for edge-coloured graphs.  We then employ RDFS in order to produce "long" rainbow paths. In particular, the main result of this section is Proposition~\ref{prop::longRainbowPath}, stated below, which can be viewed as a rainbow version of~\cite[Proposition 2.2]{Kriv}.

\medskip
\noindent
{\bf RDFS algorithm.} The input for the RDFS algorithm consists of a graph $G$ with vertex-set $[n]$, an edge-colouring $\psi : E(G) \to \mathbb{N}$ of $G$, and a permutation $\pi \in S_n$. During its execution, the algorithm maintains three sets of vertices, namely $S$, $T$ and $U$, as well as a set of colours denoted by $A_U$.  The set $S$ consists of all vertices of $G$ whose exploration is complete; the set $T$ consists of all vertices of $G$ that were not yet visited; finally, $U := [n] \setminus (S \cup T)$. The members of $U$ are kept in a stack. Given $U = \{u_1, u_2, \ldots, u_t\}$, we maintain the convention that for every $1 \leq i < j \leq t$ the vertex $u_i$ is pushed into $U$ prior to $u_j$. The set of colours $A_U$ is given by  
$$
A_U = \{\psi(u_i u_{i+1}) : 1 \leq i \leq t-1\}.
$$ 
Initially $S = U = A_U = \emptyset$ and $T = [n]$; RDFS proceeds in rounds until $T = U = \emptyset$ and $S = [n]$. In rounds for which $U = \emptyset$, RDFS chooses the first vertex in $T$ (according to $\pi$), deletes it from $T$, and pushes it into $U$. In each round for which $U \neq \emptyset$, the following actions are performed. Let $U = \{u_1, u_2, \ldots, u_t\}$ and $T = \{w_1, \ldots, w_m\}$, where the elements of $T$ are ordered according to $\pi$ (that is, $\pi(w_i) < \pi(w_j)$ if and only if $i < j$). If there exists an $i \in [m]$ such that $u_t w_i \in E(G)$ and $\psi(u_t w_i) \notin A_U$, then for the smallest such $i$ the vertex $w_i$ is deleted from $T$ and pushed into $U$. If no such $i$ is found, then $u_t$ is popped out of $U$ and added to $S$.

The following properties are maintained by RDFS. 
\begin{itemize}
\item [(D1)] In every round of the algorithm, either one vertex is moved from $T$ to $U$ or one vertex is moved from $U$ to $S$.

\item [(D2)] $|\{\psi(e) : e \in E_G(S, T)\}| \leq n-1$ holds at any point during the execution of the algorithm.

\item [(D3)] If $U = \{u_1, \ldots, u_t\}$, then $u_1 u_2 \ldots u_t$ is a path in $G$ which is rainbow under $\psi$ (in particular, $A_U$ is well-defined). 
\end{itemize}           

While Properties (D1) and (D3) are immediate, Property (D2) merits a brief explanation. Consider an arbitrary $s \in S$ at the point in time where $s$ is moved from $U$ into $S$, every edge connecting $s$ to $T$ must have its colour present in $A_U$ (where here $S, T, U$ and $A_U$ are the "snapshots" of these sets corresponding to the moment in time under examination). Hence, the set of colours $\{\psi(e) : e \in E_G(S, T)\}$ is a subset of the union of the sets $A_U$ taken over all rounds of the execution of the algorithm. Since (just like in the usual DFS algorithm) all sets $U$ span paths whose union over all rounds of the execution of the algorithm is a forest, the aforementioned union of the sets $A_U$ is of size at most $n-1$.

We are ready to state the main result of this section.

\begin{proposition} \label{prop::longRainbowPath}
Let $k < n$ be positive integers and let $B$ be a set of at least $n$ colours. Let $G$ be a graph with vertex-set $[n]$ and let $\psi : E(G) \to B$. If $|\{\psi(e) : e \in E_G(X,Y)\} | \geq n$ holds for every pair of disjoint sets $X, Y \subseteq V(G)$ of size $|X| = |Y| = k$, then $G$ admits a path of length $n - 2k$ which is rainbow under $\psi$.   
\end{proposition}

\begin{proof}
Run RDFS over $G$, $\psi$, and an arbitrary permutation $\pi \in S_n$. Consider the moment in time during the execution of the algorithm at which $|S| = |T|$; such a moment must exist by Property~(D1). Since $|\{\psi(e) : e \in E_G(S,T)\}| \leq n-1$ holds by Property (D2), it follows, by assumption, that $|S| = |T| \leq k-1$, implying that $|U| \geq n - 2k + 2$. The proof is now complete since $U$ spans a path in $G$ which is rainbow under $\psi$ by Property (D3).       
\end{proof}

\subsection{Long rainbow paths in jumbled graphs}

An $n$-vertex graph $G$ is said to be $(p,\beta)$-{\em jumbled} if 
$$
\big|e_G(X,Y) - p|X||Y| \big| \leq \beta \sqrt{|X||Y|}
$$
holds for every $X,Y \subseteq V(G)$. The canonical examples of such graphs are the so-called $(n,d,\lambda)$-{\em graphs} (see, e.g.,~\cite{KS06}) and random graphs. 
The latter, for instance, satisfy $\beta = \Theta(\sqrt{pn})$ asymptotically almost surely. More generally, $\beta \geq \sqrt{pn}$ is compelled~\cite{KS06}. Below we assume that $\beta \leq pn/D$, for some constant $D$, which in turn imposes a lower bound on $p$. Indeed, $\sqrt{pn} \leq \beta \leq pn/D$ implies that $p \geq D^2/n$.
The reason we require such a lower bound on $p$ will become apparent towards the end of this section.   

The following result asserts that randomly edge-coloured pseudorandom graphs, satisfying a fairly mild discrepancy condition, a.a.s. admit almost-spanning rainbow paths.  

\begin{theorem} \label{th::longRainbowPathPseudo}
For every $\alpha>0$ and $\eps >0 $ there exists a constant $D = D(\alpha, \eps)$ such that the following holds whenever $\beta := \beta(n) \leq pn/D$ and $n$ is sufficiently large. Let $G$ be an $n$-vertex $(p,\beta)$-jumbled graph and let $\psi$ be an edge-colouring of $G$ assigning every edge of $G$ a colour from $[(1 + \alpha) n]$, chosen independently and uniformly at random. Then a.a.s. $G$ admits a path of length $(1 - \eps) n$ which is rainbow under $\psi$.  
\end{theorem}

\begin{proof}
Given $\alpha$ and $\eps$, set $D \gg \alpha^{-1}, \eps^{-1}$. By Proposition~\ref{prop::longRainbowPath} with $k = \eps n/2$ (which we assume is an integer) and the set of colours $B = [(1 + \alpha) n]$, it suffices to show that a.a.s. $|\{\psi(e) : e \in E_G(X,Y)\}| \geq n$ holds for every pair of disjoint sets $X, Y \subseteq V(G)$ of size $|X| = |Y| = k$. In order to do so, we prove that for every subset $A \subseteq [(1 + \alpha) n]$ of size $\alpha n$ (which we assume is an integer) and every pair of disjoint sets $X, Y \subseteq V(G)$ of size $|X| = |Y| = k$ it holds that $\{\psi(e) : e \in E_G(X,Y)\} \cap A \neq \emptyset$. 
%The latter implies the desired claim; for indeed, if $|\{\psi(e) : e \in E_G(X,Y)\}| < n$ is asserted, then a set of unused colours of size at least $\alpha n$ remains which in turn meets $\{\psi(e) : e \in E_G(X,Y)\}$. 

By the jumbledness condition put on $G$, 
$$
e_G(X,Y) \geq p|X||Y| - \beta \sqrt{|X||Y|} = p k^2 -\beta k 
$$
holds for every pair of disjoint set $X,Y \subseteq V(G)$ of size $k$ each. 
In particular, for $k = \eps n/2$, we have that 
$$
p k^2 - \beta k \geq p \eps^2 n^2/4 - p \eps n^2/ (2D) \geq p \eps^2n^2/8,
$$
where the last inequality holds by our assumption that $D \gg \eps^{-1}$. 
Then, the probability that $\{\psi(e) : e \in E_G(X,Y)\} \cap A = \emptyset$ holds  
for any such pair $X,Y$ and set of colours $A$ (of size $\alpha n$) is bounded from above by 
\begin{align*}
\binom{(1 + \alpha) n}{\alpha n} \binom{n}{k}^2 \left(1 - \frac{\alpha n}{(1 + \alpha) n} \right)^{p \eps^2n^2/8} &\leq \left(\frac{e (1 + \alpha)}{\alpha} \right)^{\alpha n} \left(\frac{2 e}{\eps} \right)^{\eps n} e^{- \frac{\alpha}{8(1 + \alpha)} \cdot p\eps^2n^2} \\
& \leq \exp \left\{ 2\ln(\alpha^{-1}) \alpha n + 2 \ln (\eps^{-1}) \eps n - \alpha p\eps^2n^2/16\right\}\\ 
& = o(1).  
\end{align*}
The last equality holds since (as noted in the paragraph preceding the statement of Theorem~\ref{th::longRainbowPathPseudo}) the assumed upper bound on $\beta$ implies that $p \geq D^2/n$, and $D$ is sufficiently large with respect to $\alpha^{-1}$ and $\varepsilon^{-1}$.
\end{proof}

We conclude this section with the following direct consequence of Theorem~\ref{th::longRainbowPathPseudo} which is a rainbow version of a classical result of Ajtai, Koml\'os and Szemer\'edi~\cite{AKSz} and independently of Fernandez de la Vega~\cite{Vega}.

\begin{corollary} \label{cor::longRainbowPathGnp}
For positive constants $\alpha$ and $\eps$ there exists a constant $K = K(\alpha, \varepsilon)$ such that the following holds. Let $G \sim \mathbb{G}(n, K/n)$ and let $\psi$ be a colouring assigning every edge of $G$ a colour from $[(1 + \alpha) n]$, chosen independently and uniformly at random. Then a.a.s. $G$ admits a path of length $(1 - \eps) n$ which is rainbow under $\psi$.  
\end{corollary}

\begin{remark}
The fact that $\mathbb{G}(n, K/n)$ a.a.s. satisfies the discrepancy condition set in Theorem~\ref{th::longRainbowPathPseudo} follows by a standard application of Chernoff's bound and a union bound. One can also prove Corollary~\ref{cor::longRainbowPathGnp} directly (i.e., without relying on Theorem~\ref{th::longRainbowPathPseudo}), and such a proof avoids the use of Chernoff's bound.  
\end{remark}

\section{Rainbow Hamilton cycles in the perturbed model}

In this section we prove Theorem~\ref{th::rainbowHAM}. The main ingredients of our proof are Corollary~\ref{cor::longRainbowPathGnp}, a {\sl randomness shift} argument, taken from~\cite{KKS17}, which shifts randomness from the random perturbation to the seed, so to speak, and an absorbing structure.

\medskip
\noindent
{\bf Absorbers.} We commence with a description of the absorbing structure, which can be viewed as a rainbow variant of the one used in~\cite{HMMMO} (see also~\cite{BMPP18}). Let $H_1$ and $H_2$ be edge-disjoint graphs on the same vertex-set and let $H = H_1 \cup H_2$. Let $\psi : E(H) \to \mathbb{N}$ be an edge-colouring. Let $P = p_1 p_2 \ldots p_{\ell}$ be a path in $H_1$ which is rainbow under $\psi$ and let $A = A(P) = \{\psi(p_i p_{i+1}) : 1 \leq i \leq \ell - 1\}$ be the set of colours seen along $P$ under $\psi$. Let $I = I(P) = \{p_{2i} : 1 \leq i \leq \ell/2\} \setminus \{p_{\ell}\}$. For any two vertices $u, v \in V(H)$, set 
$$ %\begin{equation} \label{eq::absorber}
B(u,v) = \{x \in N_{H_2}(u) \cap I : N_P(x) \subseteq N_{H_2}(v)\}
$$ %\end{equation}
and put 
$$ %\begin{equation} \label{eq::rainbowAbsorber}
B^r(u,v) = \{x \in B(u,v) : |(\{\psi(ux)\} \cup \{\psi(yv) : y \in N_P(x)\}) \setminus A| = 1 + |N_P(x)|\}.
$$ %\end{equation}
%(in what follows $1 + |N_P(x)|$ is equal to $3$). 

We use sets of the form $B^r(u,v)$ in order to extend a given rainbow path by absorbing an external vertex so that the resulting extension remains rainbow. If for a  vertex $v \in V(H) \setminus P$ there exists a vertex $p_j \in B^r(p_{\ell}, v)$ (note that $p_\ell$ is an end of $P$), then there is a rainbow path in $H$ which is strictly longer than $P$. To see this, let $p_j \in B^r(p_{\ell}, v)$. Then,
$p_j  \in N_{H_2}(p_{\ell}) \cap I$ and $p_{j-1}, p_{j+1} \in N_{H_2}(v)$. Moreover, 
$$ 
\{\psi(p_{j-1} v), \psi(p_{j+1} v), \psi(p_{\ell} p_j)\} \cap A = \emptyset, 
$$
and 
$$
|\{\psi(p_{j-1} v), \psi(p_{j+1} v), \psi(p_{\ell} p_j)\}| = 1 + |N_P(p_j)| = 3.
$$ 
Therefore, the path $p_1 \ldots p_{j-1} v p_{j+1} \ldots p_{\ell} p_j$ forms a rainbow path in $H$ with vertex-set $V(P) \cup \{v\}$. We say that $p_j$ was \emph{used to absorb} $v$.

\medskip 
\noindent
{\bf Randomness shift.} Next, we describe the \emph{randomness shift} argument. Let $H$ be a graph with vertex-set $[n]$ and let $R \sim \mathbb{G}(n,p)$. Suppose that a.a.s. $R$ contains certain substructures. Then we may assume that these substructures (or some corresponding vertex-sets) are themselves sampled uniformly at random. Indeed, $R$ can be generated as follows. First, generate a random graph $R' \sim \mathbb{G}(n,p)$ and then permute its vertex-set {\sl randomly};  denote the resulting graph by $R$. That is, we choose a permutation $\pi \in S_n$ uniformly at random and set $R = ([n], \{\pi(u) \pi(v) : uv \in E(R')\})$. The corresponding probability space coincides with $\mathbb{G}(n,p)$. In this manner, the aforementioned substructures of $R$ are sampled uniformly at random through $\pi$. 
Below we apply this argument to an almost-spanning rainbow path.

\medskip

We are now ready to prove our main result, namely, Theorem~\ref{th::rainbowHAM}.

\begin{proof} [Proof of Theorem~\ref{th::rainbowHAM}]
Let $d$ and $\alpha$ be as in the premise of the theorem and set $\eps = d^3/220$. Let $K = K(\alpha, \eps)$ be the constant whose existence is ensured by Corollary~\ref{cor::longRainbowPathGnp} and let $C = K+1$. We expose $R \sim \mathbb{G}(n, C/n)$ in two rounds, that is, $R = R_1 \cup R_2$ where $R_1 \sim \mathbb{G}(n, K/n)$ and $R_2 \sim \mathbb{G}(n, p)$ for $p$ which satisfies $1 - C/n = (1 - K/n)(1-p)$; note that $p \geq 1/n$.

We first expose the edges of $R_1$ and colour them uniformly at random with colours from $[(1 + \alpha) n]$; denote the resulting colouring by $\psi$. Set $\ell = (1 - \eps) n$ (which we assume is an integer) and let $P = p_1 p_2 \ldots p_{\ell}$ be a path in $R_1$ which is rainbow under $\psi$; such a path exists (a.a.s. in $R_1$) by Corollary~\ref{cor::longRainbowPathGnp}.

Next, we use the edges of $H \in \G_{d,n}$ in order to extend $P$ to a rainbow path on $n-2$ vertices. Let $I = \{p_{2i} : 1 \leq i \leq \ell/2\} \setminus \{p_{\ell}\}$. We begin by proving that, with respect to $H$, $P$ and $I$, the set $B(u,v)$ is large for every $u, v \in V(H)$; this is done without revealing the colours of the edges of $H$.

\begin{claim} \label{cl::LargeBuv}
Asymptotically almost surely $|B(u,v)| \geq d^3 n/110$ holds for every $u, v \in V(H)$.
\end{claim}

\begin{proof}
Fix some $u, v \in V(H)$. As explained above, we may assume that a random permutation $\pi : V(R_1) \to V(H)$ maps $P$ to a path $P'$. We assume that the images $\pi(p_1), \pi(p_2), \ldots, \pi(p_{\ell})$ are determined (randomly) first in this order, and then the images $\pi(w)$ are set for every $w \in V(R_1) \setminus V(P)$ in an arbitrary order. For every $1 \leq i \leq \ell$, let $X_i$ denote the indicator random variable for the event $\pi(p_i) \in N_H(u)$ and let $Y_i$ denote the indicator random variable for the event $\pi(p_i) \in N_H(v)$. For every $i$ such that $p_i \in I$, let $Z_i = Y_{i-1} X_i Y_{i+1}$. Then $Z_i$ is the indicator random variable for the event $\pi(p_i) \in B(u,v)$ and thus $|B(u,v)| = \sum Z_i$, where the sum is extended over all $1 \leq i \leq \ell$ for which $p_i \in I$. Let $A_u$ (respectively $A_v$) be the event that $|N_H(u) \cap \{\pi(p_1), \ldots, \pi(p_{n/3})\}| \geq |N_H(u)|/2$ (respectively $|N_H(v) \cap \{\pi(p_1), \ldots, \pi(p_{n/3})\}| \geq |N_H(v)|/2$). Note that $|N_H(u) \cap \{\pi(p_1), \ldots, \pi(p_{n/3})\}|$ (and its counterpart for $v$) is distributed hypergeometrically owing to the randomness shift argument. A straightforward application of Chernoff's bound for the hypergeometric distribution then shows that $\mathbb{P}(A_u \cup A_v) = o(1)$. Hence, for the remainder of the proof we will assume that $A^c_u \cap A^c_v$ holds.

Observe that 
$$
\mathbb{P}(Z_{4i} = 1) \geq \frac{|N_H(v)| - \sum_{j=1}^{4i-2} Y_j}{n} \cdot \frac{|N_H(u)| - \sum_{j=1}^{4i-1} X_j}{n} \cdot \frac{|N_H(v)| - \sum_{j=1}^{4i} Y_j}{n} \geq \frac{d^3}{9}.
$$
holds for every $1 \leq i \leq n/12$, regardless of the value of $Z_{4j}$ for any $j \neq i$. Therefore,
$$
\mathbb{P}(|B(u,v)| < d^3 n/110) \leq \mathbb{P}(\Bin(n/12, d^3/9) < d^3 n/110) < e^{- \Omega(d^3 n)},
$$
where the last inequality holds by a standard application of Chernoff's bound. Finally, a union bound over all pairs $u, v \in V(H)$ shows that the probability that there exists such a pair for which $|B(u,v)| < d^3 n/110$ is $o(1)$.
\end{proof}

Let $P_0 = P$ (formally, $P_0 = \pi(P)$, but we avoid using this more accurate notation for the sake of clarity of the presentation) and let $x, y, v_1, \ldots, v_s$ be the vertices of $V(H) \setminus V(P_0)$. We extend $P_0$ (via the edges of the seed $H$) by absorbing $v_1, \ldots, v_s$ one by one whilst keeping $p_1$ as one of the ends throughout. Assume that for some $i \geq 0$ the path $P_i$ with vertex-set $V(P_0) \cup \{v_1, \ldots, v_i\}$ has already been built and consider the subsequent extension of $P_i$ into $P_{i+1}$, obtained by absorbing $v_{i+1}$. For every $1 \leq j \leq i$, let $u_j$ be the vertex of $I$ that was used to absorb $v_j$. Let $z$ denote the endpoint of $P_i$ that is not $p_1$ (note that $z = p_{\ell}$ if $i = 0$ and $z = u_i$ otherwise). 
%For $v_1$, the vertex $u_1$ is defined only after $v_1$ is absorbed naturally, hence the absorption of $v_1$ differs slightly in this regard and it does not require $u_1$. 

Let $B_i(z,v_{i+1}) = B(z,v_{i+1}) \setminus \{u_1, \ldots, u_i\}$ (in particular, $B_0(z,v_{i+1}) = B(z,v_{i+1})$), and note that  
\begin{equation} \label{eq:Bi}
|B_i(z,v_{i+1})| \geq |B(z,v_{i+1})| - i \geq d^3 n/110 - \eps n \geq d^3 n/220
\end{equation}
holds for any $0 \leq i \leq s$, by Claim~\ref{cl::LargeBuv} and the choice of $\eps$.
This removal of vertices that were previously used for absorption is crucial in two respects. First, absorbing triples cannot be reused. Second, and this unfolds more explicitly towards the end of the proof, there is a need to keep track over edges for which the random colouring $\psi$ has already been exposed and where randomness still lies, so to speak.  

For every vertex $p_j \in B_i(z,v_{i+1})$, we now expose the colours of the edges $p_j z$, $p_{j-1} v_{i+1}$, and $p_{j+1} v_{i+1}$ and extend the colouring $\psi$ to these edges; note that, crucially, the colours of $E_H(v_{i+1}, P_i) \cup E_H(z, I)$ were not previously exposed. An exception to this rule occurs if one of these edges is in $R_1$ as well. However, a standard calculation shows that a.a.s. $\Delta(R_1) = o(\ln n)$ implying that a.a.s. $e_{R_1}(v_{i+1}, P_i) + e_{R_1}(z, I) \leq \ln n$. If there exists a vertex $p_j \in B_i(z,v_{i+1}) \cap B^r(z,v_{i+1})$, then it can be used to absorb $v_{i+1}$ as explained above. The probability that no such vertex exists is at most
$$
\left(1 - \left(\frac{\alpha}{1 + \alpha} \right)^3\right)^{|B_i(z,v_{i+1})| - \ln n} \leq e^{- \alpha^3 d^3 n/230} = o(1/n).
$$    
Since this holds for every $0 \leq i \leq s$, a union bound shows that the probability that we fail to absorb at least one of the vertices $v_1, \ldots, v_s$ is o(1).

\medskip

Denote the ends of the resulting rainbow path $P_s$ by $p_1$ and $p_{n-2}$. We now use the edges of $R_2$ in order to extend $P_s$ into a rainbow Hamilton cycle. Let $X = B_s(p_1, x)$ and let $Y = B_s(p_{n-2}, y)$. The sizes of these sets are captured by~\eqref{eq:Bi}. Note that the colours of $E_H(\{x,y\}, V(H))$ were not yet exposed. Similarly, the colours of $E_H(p_1, X) \cup E_H(p_{n-2}, Y)$ were not yet exposed. Indeed, neither $X$ nor $Y$ meet the set $\{u_1, \ldots, u_s\}$, as by definition of the $B_i$-sets, members of $\{u_1, \ldots, u_s\}$ are repeatedly removed. Moreover, as the $B_i$-sets start from sets that do not meet the set $\{v_1, \ldots, v_s\}$ and in subsequent absorption steps are only refined, neither $X$ nor $Y$ meet the set $\{v_1, \ldots, v_s\}$. The claim follows by observing that throughout the absorption process the sole edges of $H$ whose colour is exposed are incident with $\{p_{\ell}, u_1, \ldots, u_s\} \cup \{v_1, \ldots, v_s\}$. Other relevant edges whose colour was already exposed, are the edges of $E_{R_1}(X, Y)$ and the edges of $E_{R_1}(x, P_s) \cup E_{R_1}(y, P_s) \cup E_{R_1}(p_1, X) \cup E_{R_1}(p_{n-2}, Y)$. A standard application of Chernoff's bound and a union bound over all pairs of sets of appropriate sizes shows that a.a.s. $e_{R_1}(X,Y) \leq 2 K |X||Y|/n = o(|X||Y|)$. Moreover, a standard calculation shows that a.a.s. $e_{R_1}(x, P_s) + e_{R_1}(y, P_s) + e_{R_1}(p_1, X) + e_{R_1}(p_{n-2}, Y) \leq \ln n$.  
 
%on the original path $P_0$ and those incident to the sets $\{u_1,\ldots,u_s\}$ and $\{v_1,\ldots,v_s\}$. Furthermore, in terms of $P_s$, there are two non-exclusive types of edges whose colour has been exposed; those present on $P_s$ and those incident with $\{u_1,\ldots,u_s\} \cup\{v_1,\ldots,v_s\}$. It follows that none of the members of $E_H(X,Y) \setminus E(P_S)$, if non-empty, has its colour exposed, where by $E_H(X,Y)$ we mean the set of edges in $H$ with one end in $X$ and the other in $Y$ ($X \cap Y \neq \emptyset$ possible).

%As $x$ and $y$ have not played a part in the absorption process thus far none of the edges (of $H$) incident to them has its colour under $\psi$ exposed. In particular, all edges (in $H$) of the form $(x,X)$ and $(y,Y)$ have their colour still unexposed. 

%The same holds for edges of the form $(p_1,X)$ and $(p_{n-2},Y)$.    

%We may assume that $E_H(X,Y) \setminus E(P_s) = \emptyset$, for otherwise it only works in our favour. 

Let $X' := \{p_i \in X : \{p_i p_1, p_{i-1} x, p_{i+1} x\} \cap E(R_1) = \emptyset\}$ and $Y' := \{p_i \in Y : \{p_i p_{n-2}, p_{i-1} y, p_{i+1} y\} \cap E(R_1) = \emptyset\}$; as noted in the preceding paragraph $|X'| \geq |X| - \ln n$ and $|Y'| \geq |Y| - \ln n$ hold asymptotically almost surely. Expose the edges of $R_2$ with one endpoint in $X'$ and the other in $Y'$ (that are not edges of $R_1$) and extend the colouring $\psi$ to these edges, and to the edges in $E_H(x, P_s) \cup E_H(y, P_s) \cup E_H(p_1, X') \cup E_H(p_{n-2}, Y')$ (that are not edges of $R_1$). If there exists an edge $p_i p_j \in E_{R_2}(X', Y') \setminus E(R_1)$ such that 
$$
p_i \in B^r(p_1, x),\;\; p_j \in B^r(p_{n-2}, y),\;\; \psi(p_i p_j) \notin \{\psi(p_t p_{t+1}) : 1 \leq t \leq n-3\},  
$$
and 
$$
|\{\psi(p_i p_j), \psi(p_1 p_i), \psi(x p_{i-1}), \psi(x p_{i+1}), \psi(p_{n-2} p_j), \psi(y p_{j-1}), \psi(y p_{j+1})\}| = 7,
$$ 
then (assuming without loss of generality that $i < j$) the sequence 
$$
p_i p_1 \ldots p_{i-1} x p_{i+1} \ldots p_{j-1} y p_{j+1} \ldots p_{n-2} p_j p_i
$$ 
forms a Hamilton cycle of $H \cup R$ which is rainbow under $\psi$. 

Since $P_s$ is coloured using $n-3$ colours, the colours for the aforementioned seven edges (assuming $p_i p_j$ exists) can be chosen from a set of at least $\alpha n$ colours. The probability that such an edge $p_i p_j$ does not exist in $R_2$ or that it does exist in $R_2$ and a colour clash occurs along the aforementioned seven edges as detailed above, is at most 
$$
\left(1 - \frac{O_\alpha(1)}{n}  \right)^{|X'| |Y'| - |X' \cap Y'|^2/2 - e_{R_1}(X', Y')} \leq \left(1 - \frac{O_\alpha(1)}{n}  \right)^{|X'| |Y'|/3} \leq e^{- \Omega_{\alpha,d}(n)} = o(1),
$$
where the second inequality holds by~\eqref{eq:Bi}.
\end{proof}

\bibliographystyle{amsplain}
\bibliography{lit}

\end{document}